\documentclass[11pt]{amsart}
\usepackage{amsfonts,amssymb,amscd,amsmath,enumerate,verbatim,calc}

\def\NZQ{\mathbb}               
\def\NN{{\NZQ N}}

\def\ZZ{{\NZQ Z}}

\def\frk{\mathfrak}               

\def\mm{{\frk m}}

\def\Phi{{\frk N}}

\def\opn#1#2{\def#1{\operatorname{#2}}} 
\opn\chara{char} \opn\length{\ell} \opn\pd{pd} \opn\rk{rk}
\opn\projdim{proj\,dim} \opn\injdim{inj\,dim} \opn\rank{rank}
\opn\depth{depth} \opn\grade{grade} \opn\height{height}
\opn\embdim{emb\,dim} \opn\codim{codim}

\opn\Tr{Tr} \opn\bigrank{big\,rank}
\opn\superheight{superheight}\opn\lcm{lcm}
\opn\trdeg{tr\,deg}
\opn\reg{reg} \opn\lreg{lreg} \opn\ini{in} \opn\lpd{lpd}
\opn\size{size}\opn{\mult}{mult}
\opn\div{div} \opn\Div{Div} \opn\cl{cl} \opn\Cl{Cl}
\opn\Spec{Spec} \opn\Supp{Supp} \opn\supp{supp} \opn\Sing{Sing}
\opn\Ass{Ass} \opn\Min{Min}
\opn\Ann{Ann} \opn\Rad{Rad} \opn\Soc{Soc}
\opn\Syz{Syz} \opn\Im{Im} \opn\Ker{Ker} \opn\Coker{Coker}
\opn\Am{Am} \opn\Hom{Hom} \opn\Tor{Tor} \opn\Ext{Ext}
\opn\End{End} \opn\Aut{Aut} \opn\id{id}

\opn\nat{nat}
\opn\pff{pf}
\opn\Pf{Pf} \opn\GL{GL} \opn\SL{SL} \opn\mod{mod} \opn\ord{ord}
\opn\Gin{Gin}
\opn\Hilb{Hilb}\opn\adeg{adeg}\opn\std{std}\opn\ip{infpt}
\opn\Pol{Pol}
\opn\sat{sat}
\opn\Var{Var}
\opn \ann{ann}
\opn\sdepth{sdepth}
\opn\aff{aff} \opn\con{conv} \opn\relint{relint} \opn\st{st}
\opn\lk{lk} \opn\cn{cn} \opn\core{core} \opn\vol{vol}
\opn\link{link} \opn\star{star}
\opn\gr{gr}

\def\pot#1#2{#1[\kern-0.28ex[#2]\kern-0.28ex]}

\opn\dirlim{\underrightarrow{\lim}}
\opn\inivlim{\underleftarrow{\lim}}
\let\iso=\cong
\let\Union=\bigcup

\let\Dirsum=\bigoplus

\let\To=\longrightarrow
\def\Implies{\ifmmode\Longrightarrow \else
        \unskip${}\Longrightarrow{}$\ignorespaces\fi}
\def\implies{\ifmmode\Rightarrow \else
        \unskip${}\Rightarrow{}$\ignorespaces\fi}
\def\iff{\ifmmode\Longleftrightarrow \else
        \unskip${}\Longleftrightarrow{}$\ignorespaces\fi}

\let\:=\colon
\theoremstyle{plain}
\newtheorem{Theorem}{Theorem}[section]
\newtheorem{Lemma}[Theorem]{Lemma}
\newtheorem{Corollary}[Theorem]{Corollary}
\newtheorem{Proposition}[Theorem]{Proposition}
\newtheorem{Remark}[Theorem]{Remark}

\let\epsilon\varepsilon
\let\phi=\varphi
\let\kappa=\varkappa

\def\qed{\ifhmode\textqed\fi
      \ifmmode\ifinner\quad\qedsymbol\else\dispqed\fi\fi}
\def\textqed{\unskip\nobreak\penalty50
       \hskip2em\hbox{}\nobreak\hfil\qedsymbol
       \parfillskip=0pt \finalhyphendemerits=0}
\def\dispqed{\rlap{\qquad\qedsymbol}}

\opn\dis{dis}
\def\pnt{{\raise0.5mm\hbox{\large\bf.}}}

\opn\Lex{Lex} \opn\hreg{hreg}

\begin{document}

\title{Skeletons of monomial ideals}

\author{J\"urgen Herzog, Ali Soleyman Jahan and  Xinxian Zheng}

\address{J\"urgen Herzog, Fachbereich Mathematik und
Informatik, Universit\"at Duisburg-Essen, Campus Essen, 45117
Essen, Germany} \email{juergen.herzog@uni-due.de}

\address{Ali Soleyman Jahan , Fachbereich Mathematik und
Informatik, Universit\"at Duisburg-Essen, Campus Essen, 45117
Essen, Germany} \email{ali.soleyman-jahan@stud.uni-due.de}

\address{Xinxian Zheng, Fachbereich Mathematik und
Informatik, Universit\"at Duisburg-Essen, Campus Essen, 45117
Essen, Germany} \email{xinxian.zheng@uni-due.de}

\thanks{The third author is grateful for the financial support by DFG (Deutsche Forschungsgemeinschaft)
during the preparation of this work}

\begin{abstract}
In analogy to the skeletons of a simplicial complex and their
Stanley--Reisner ideals  we introduce the skeletons of an
arbitrary monomial ideal $I\subset S=K[x_1,\ldots,x_n]$.  This
allows us to compute the depth of $S/I$ in terms of its skeleton
ideals.  We apply these techniques to show that Stanley's
conjecture on Stanley decompositions of $S/I$ holds provided it
holds whenever $S/I$ is Cohen--Macaulay. We also discuss a
conjecture of Soleyman-Jahan and show that it suffices to prove
his conjecture for monomial ideals with linear resolution.
\end{abstract}
\maketitle
\section*{Introduction}

Let $\Delta$ be a simplicial complex of dimension $d-1$  on the vertex set $\{1,\ldots,n\}$, $K$ a field and
$K[\Delta]$ the Stanley--Reisner ring of $\Delta$. The depth of $K[\Delta]$ can be expressed in terms of the skeletons
of $\Delta$, as has been shown by D.\ Smith \cite[Theorem 3.7]{Sm} for pure simplicial complexes, and by Hibi
\cite[Corollary 2.6]{Hi} in general.  The $j$th skeleton of $\Delta$ is the simplicial subcomplex
$\Delta^{(j)}=\{F\in\Delta\:\; |F|\leq j\}$ of $\Delta$. The result is that
$\depth K[\Delta]=\max\{j\:\; \text{$\Delta^{(j)}$ is Cohen--Macaulay}\}$.

The purpose of this paper is to generalize this result as follows:
first  note that we have the following chain of Stanley--Reisner
ideals $I_\Delta=I_{\Delta^{d}}\subset I_{\Delta^{d-1}}\subset
\cdots \subset I_{\Delta^{0}}\subset S$ with $\dim
S/I_{\Delta^{(j)}}=j$  for all $j$. Now for an arbitrary monomial
ideal $I\subset S$ we want to define in a natural way  a similar
chain of monomial ideals $I=I_d\subset I_{d-1}\subset \cdots
\subset I_0\subset S$ with $\dim S/I_j=j$  for all $j$,
and of course this chain should satisfy the condition that $\depth
S/I=\max\{j\:\; \text{$S/I_j$ is Cohen--Macaulay}\}$. We show
in Section~1 that such a natural chain of monomial ideals with
these properties indeed exists. The ideal $I_j$ is called the
$j$th {\em skeleton ideal} of $I$.

For the construction of the skeleton ideals of $I$ we consider the
so-called {\em characteristic poset} $P^g_{S/I}$ introduced in
\cite{HVZ}. Here $g\in\NN^n$ is an integer vector such that $g\geq
a$ for all $a$ for which  $x^a$ belongs to the minimal set of
monomial generators of $I$, and  $P^g_{S/I}$ is the (finite) poset
of all $b\in \NN^n$ such that $b\leq g$ and $x^b\not\in I$. Here the partial order on $\NN^n$ is defined as follows:   $a\leq b$ if and only if  $a(i)\leq b(i)$ for $i=1,\ldots,n$. In case of a
Stanley--Reisner ideal $I_\Delta$ and $g=(1,1,\,\ldots,1)$ this
poset is just the face poset of $\Delta$. For each $b\in \NN^n$,
let $\rho(b)=|\{j\:\; b(j)=g(j)\}|$. It  has been shown in
\cite[Corollary 2.6]{HVZ} that $\dim S/I=\max\{\rho(b)\:\; b\in
P^g_{S/I}\}$. We use this integer function $\rho$ to define the
skeleton ideals of $I$, and let  $I_j$ be the monomial  ideal
generated by $I$ and all $x^b$ with $\rho(b)>j$. It is the easy to
see that $\dim S/I_j=j$ for all $j$. The crucial result
however is that for all $j$, $I_{j-1}/I_j$ is Cohen-Macaulay
module of dimension $j$, see Theorem~\ref{crucial}. From this
result we easily deduce in Corollary ~\ref{hibi} a generalization
of the result of Hibi, namely that $\depth S/I=\max\{j\:\;
\text{$S/I_j$ is Cohen--Macaulay}\}$.

In Section~2 we apply the results and techniques introduced in Section~1 to deduce some results on Stanley decomposition.
Let $M$ be a finitely generated $\ZZ^n$-graded $S$-module,  $m\in M$ be a homogeneous element and
$Z\subset X=\{x_1,\ldots,x_n\}$. We denote by $mK[Z]$ the
$K$-subspace of $M$ generated by all homogeneous elements of the
form $mu$, where $u$ is a monomial in $K[Z]$. The $K$-subspace
$mK[Z]$ is called a {\em Stanley space of dimension $|Z|$} if
$mK[Z]$ is a free $ K[Z]$-module.

A decomposition $\mathcal D$ of $M$ as a finite direct sum of
Stanley spaces is called a {\em Stanley decomposition} of $M$. The
minimal dimension of a Stanley space in the decomposition
$\mathcal D$ is called the {\em Stanley depth} of $\mathcal D$,
denoted by $\sdepth {\mathcal D}$. We set
$$\sdepth M=\max\{\sdepth{\mathcal D}\: {\mathcal D}\; \text{is a
Stanley decomposition of $M$}\}, $$ and call this number the {\em
Stanley depth} of $M$. A famous conjecture of Stanley asserts that
$\sdepth M\geq \depth M$.

As one of the main results of Section~2 we show in
Corollary~\ref{reduction} that for each monomial ideal $I$
Stanley's conjecture holds for $S/I$ if it holds whenever $S/I$ is
Cohen--Macaulay. We also discuss a conjecture of Soleyman-Jahan.
His conjecture asserts that we can always find a Stanley
decomposition $M=\Dirsum_{j=1}^rm_jK[Z_j]$ of $M$ with $|\deg
m_j|\leq \reg(M)$ for all $j$. Here $|a|=\sum_{i=1}^na(i)$ for
$a\in\ZZ^n$. We show in the case that $M=I$ is a monomial ideal,
it suffices to prove this conjecture when $I$ has a linear
resolution.

\section{Characteristic posets and skeletons}
Let $K$ be a field, $S=K[x_1,\ldots,x_n]$ the polynomial ring in $n$ variables and
 $I\subset S$  a monomial ideal. We denote by $G(I)$ the unique minimal set of monomial
 generators of $I$. Let $G(I)=\{u_1,\ldots,u_m\}$ with $u_i=x^{a_i}$ and $a_i\in \NN^n$.
 Here, for any $c\in \NN^n$ we denote as usual by  $x^c$ the monomial $x_1^{c(1)}x_2^{c(2)}\cdots x_n^{c(n)}$.

Observe that $\NN^n$ with  the natural
partial order introduced in the introduction  is a distributive lattice with meet $a\wedge  b$  and join $a\vee b$
defined as follows: $(a\wedge  b)(i)=\min\{a(i),b(i)\}$ and $(a\vee  b)(i)=\max\{a(i),b(i)\}$.
We also denote by $\epsilon_j$ the $j$th canonical unit vector in $\ZZ^n$.

Let $J\subset S$ be another monomial ideal with $I\subset J$, minimally generated by $x^{b_1},\ldots, x^{b_s}$.
We choose  $g\in\NN^n$ such that $a_i\leq g$ and $b_j\leq g$ for all $i$ and $j$ , and let
$P^g_{J/I}$ be the set of all $b\in \NN^n$ with $b\leq g$, $b\geq b_j$ for some $j$, and
$b\not\geq a_i$ for all $i$. The  set $P^g_{J/I}$ viewed as a
subposet of $\NN^n$ is a finite poset, and is called the {\em
characteristic poset} of $J/I$ with respect to $g$, see
\cite{HVZ}.

For any $b\in \NN^n$ we define subsets $Y_b=\{x_j\:\; b(j)\neq
g(j)\}$ and $Z_b=\{x_j\:\; b(j)=g(j)\}$ of $X=\{x_1,\ldots,x_n\}$
and set $\rho(b)=|\{j\:\; b(j)=g(j)\}|=|Z_b|$. Let $d=\dim J/I$ be
the Krull dimension of $J/I$. It is shown in \cite[Corollary
2.6]{HVZ} that
\begin{eqnarray}
\label{1}
d=\max\{\rho(b)\: b\in P^{g}_{J/I}\}.
\end{eqnarray}
As a consequence of (\ref{1}) we obtain
\begin{Lemma}
\label{easy}
Let $d=\dim J/I$. Then $\rho(b)\leq d$ for all $b\in \NN^n$ with   $x^b\in J\setminus I$.
\end{Lemma}

\begin{proof}
Let $g'=g\vee b$. Then $b \in P^{g'}_{J/I}$, and hence  $|\{j\:\; b(j)=g'(j)\}|\leq d$, by (\ref{1}).
Since $\rho(b)=|\{j\:\; b(j)=g(j)\}|\leq |\{j\:\; b(j)\geq g(j)\}|=|\{j\:\; b(j)=g'(j)\}|$, the assertion follows.
\end{proof}

Formula (\ref{1})  leads us to consider for each $j\leq d$, the
monomial ideal $I_j$ generated by  $I$ together with all
monomials $x^b$ such that  $\rho(b)>j$. We then obtain a chain of
monomial ideals
\[
I=I_d\subset I_{d-1}\subset \cdots \subset I_{0}\subset S.
\]
Of course this chain of ideals depends not only  on $I$, but also on the choice of $g$.

Consider the special case, where $I=I_\Delta$ is the
Stanley--Reisner ideal of a simplicial complex $\Delta$ on the
vertex set $\{1,\ldots,n\}$. Then for $g=(1,\ldots,1)$ we have
$I_j=I_{\Delta^{(j)}}$.  This observation justifies to call
$I_j$ the {\em $j$th skeleton ideal} of $I$ (with respect to
$g$).

\medskip
The following result is crucial for this note.

\begin{Theorem}
\label{crucial} For each $0\leq j\leq d$, the factor module
$I_{j-1}/I_j$ is a direct sum of cyclic Cohen--Macaulay
modules of dimension  $j$. In particular, $I_{j-1}/I_j$ is a
$j$-dimensional Cohen--Macaulay module.
\end{Theorem}

\begin{proof}
Replacing $I$ by $I_j$ it suffices to consider the case $j=d$.
Let
$$J=(I,\{x^b\:\; b\in A\}),\quad \text{where}\quad A=\{b\in P_{S/I}^g\:\; \rho(b)=d\},$$
then $I_{d-1}/I_d=J/I$.

Let $\{Z_1,\ldots,Z_r\}$ be the collection of those subsets of $X$
with the property that for each $i=1,\ldots,r$ there exists $b\in
A$ such that $Z_i=Z_b$. Let $A_i=\{b\in A\:\; Z_b=Z_i\}$, and let
$b,b'\in A_i$. Then $b\wedge b'\in A_i$. Thus the meet of all the
elements in $A_i$ is the unique smallest element in $A_i$. We
denote this element by $b_i$.  Then $Z_i=Z_{b_i}$. Obviously the
elements $f_i=x^{b_i}+I$, $i=1,\ldots,r$  generate $J/I$. We claim
that
\[
J/I=\Dirsum_{i=1}^rSf_i.
\]
The cyclic module  $Sf_i$ is $\ZZ^n$-graded with a $K$-basis   $x^a+I$ with $a\geq b_i$ and $x^a\not\in I$.
Given  $c\in\NN^n$ with   $c\geq b_i$
and $c\geq b_j$ for some $1\leq i<j\leq r$,  then  $\rho(c)>d$, and so $x^c\in I$, by Lemma~\ref{easy}.
This shows that the sum of the cyclic modules $Sf_i$ is indeed direct.

Next we  notice that if $x^c=x^{c_1}x^{c_2}$ with  $x^{c_1}\in
K[Z_{b_i}]$ and $x^{c_2}\in K[Y_{b_i}]$ belongs to $\Ann(Sf_i)$,
then $x^{c_2}\in \Ann(Sf_i)$. Indeed, $x^c=x^{c_1}x^{c_2}\in
\Ann(Sf_i)$ if and only if ${a_j}\leq {b_i+c_1+c_2}$ for some $j$.
Since $c_1(k)=0$ for all $k$ with $x_k\in Y_{b_i}$, it follows
that $a_j(k)\leq (b_i+c_2)(k)$ for all $k\in Y_{b_i}$, while for
$k$ with $x_k\in Z_{b_i}$ we have $a_j(k)\leq
g(k)=b_i(k)=(b_i+c_2)(k)$. Hence $a_j\leq b_i+c_2$, which implies
that $x^{c_2}\in \Ann(Sf_i)$.

It follows that  $\Ann(Sf_i)$  is
generated by monomials in $K[Y_{b_i}]$. In other words, there
exists a monomial ideal $M_i\subset K[Y_{b_i}]$ such that
$\Ann(Sf_i)=M_iS$.

For each $k$ with $x_k\in Y_{b_i}$ we have $b_i(k)<g(k)$ and
$\rho(b_i+(g(k)-b_i(k))\epsilon_k)=d+1$. Therefore Lemma~\ref{easy} implies that
$x^{b_i}x_k^{g(k)-b_i(k)}\in I$. It follows that
$x_k^{g(k)-b_i(k)}\in M_i$ for all $k$ with $x_k\in Y_{b_i}$.
Hence we see  that $\dim K[Y_{b_i}]/M_i=0$. This implies that  $Sf_i=S/M_iS$ is Cohen--Macaulay of dimension $d$.
\end{proof}

\begin{Remark}{\em  It is also possible to define skeletons of $J/I$ in the same way as for $S/I$, and one obtains
a chain of ideals $I=I_d\subset I_{d-1}\subset \cdots \subset I_0\subset J$. Some of the factor modules $I_{j-1}/I_j$ however may be $0$ in this generality. But whenever $I_{j-1}/I_j\neq 0$ it follows again that $I_{j-1}/I_j$ is Cohen--Macaulay of dimension $j$, though not always a direct sum of cyclic modules.}
\end{Remark}

\begin{Corollary}\label{notimportant} For $j=0\ldots,d-1$ and  $i=0\ldots,j-1$, we have $\depth(I_j/I)\geq j+1$, and
$H^i_{\mm}(S/I)\iso  H^i_{\mm}(S/I_j)$.
\end{Corollary}

\begin{proof}  We prove the assertion  by induction on $d-j$. For $j=d-1$ the assertion follows from
Theorem~\ref{crucial}. Let $j<d-1$. Then the  exact sequence
$$0\To I_{j+1}/I\To I_j/I\To I_j/I_{j+1}\To 0$$ implies that
\[
\depth(I_j/I)\geq\min\{\depth(I_{j+1}/I),\depth(I_j/I_{j+1})\},
\]
see \cite[Proposition 1.2.9]{BH}. By Theorem~\ref{crucial},
$\depth(I_j/I_{j+1})=j+1$ and by induction hypothesis
$\depth(I_{j+1}/I)\geq j+2$. Hence $\depth (I_j/I)\geq j+1$.

The short exact sequence
\[
0\To I_j/I\To S/I\To S/I_j\To 0
\]
yields the long exact sequence
\[
\cdots \To H^{i}_\mm(I_j/I)\To H^i_\mm(S/I)\To
H^i_\mm(S/I_j)\To H^{i+1}_\mm(I_j/I)\To \cdots
\]
of local cohomology. By the first part of the statement we have
$H^k_\mm(I_j/I)=0$ for $k\leq j$. This yields the desired
isomorphisms.
\end{proof}

As an application of Theorem~\ref{crucial} we obtain the following
characterization of the depth of $S/I$ which generalizes  a
classical result of Hibi   \cite[Corollary 2.6]{Hi}.

\begin{Corollary}
\label{hibi} Let $I\subset S$ be a monomial ideal. Then
\[
\depth S/I=\max\{j\:\; \text{$S/I_j$ is Cohen--Macaulay}\},
\]
and $S/I_j$ is Cohen--Macaulay for all $j\leq \depth S/I$.
\end{Corollary}

\begin{proof}
Let $d=\dim S/I$ and $t=\depth S/I$. Since
$I_j=(I_{d-1})_j$ for $j\leq d-1$, both assertions
follow by induction on $d$ once we can show the following:
\begin{enumerate}
\item[(i)]  If $t<d$, then $\depth S/I_{d-1}=t$.
\item[(ii)] If $S/I$ is Cohen--Macaulay, then $S/I_{d-1}$ is Cohen--Macaulay.
\end{enumerate}
Proof of (i): The exact sequence
\[
0\To I_{d-1}/I\To S/I\To S/I_{d-1}\To 0
\]
implies that
\begin{eqnarray}
\label{depth} \depth S/I_{d-1}\geq \min\{\depth (I_{d-1}/I)-1
, \depth S/I\},
\end{eqnarray}
with equality if $t<d-1$, see \cite[Proposition 1.2.9]{BH}.  By
Theorem~\ref{crucial}, $\depth (I_{d-1}/I)-1=d-1$.  It follows
that $\depth S/I_{d-1}=t$, if $t<d-1$. On the other hand, if
$t=d-1$, then $\depth S/I_{d-1}\geq d-1$. However, since $\dim
S/I_{d-1}= d-1$, we again get $\depth S/I_{d-1}= d-1=t$.

Proof of (ii):  If $S/I$ is Cohen--Macaulay, then $\depth S/I=d$.
Hence Theorem~\ref{crucial} and inequality (\ref{depth}) imply
that $\depth S/I_{d-1}\geq d-1$. Since $\dim S/I_{d-1}=d-1$,
the assertion follows.
\end{proof}

The proof of Corollary~\ref{hibi} provides the following additional information.

\begin{Corollary}\label{extra} We have
$ \depth S/I_{j-1}\leq\depth S/I_{j}\leq\depth S/I$ for all
$0\leq j\leq \dim S/I$.
\end{Corollary}

\section{Applications to Stanley decompositions}

Let $I\subset S$ be a monomial ideal. In the recent paper \cite{HVZ}
it was shown that the Stanley depth of $S/I$ can be computed by
means of properties of $P^g_{S/I}$. The result \cite[Theorem
2.1]{HVZ} can be summarized as follows: given any poset $P$ and
$a,b\in P$, we set $[a,b]=\{c\in P\:\; a\leq c\leq b\}$ and call
$[a,b]$ an {\em interval}. Of course, $[a,b]\neq \emptyset$ if and
only if $a\leq b$. Suppose $P$ is a finite poset. A {\em
partition} of $P$ is a disjoint union
\[
\mathcal{P}\:\; P=\Union_{i=1}^r[c_i,d_i]
\]
of non-empty intervals. Let $\mathcal{P}\:\; P^g_{S/I}=\Union_{i=1}^r[c_i,d_i]$ be a partition of $P^g_{S/I}$.
We set
\[\rho(\mathcal{P})=\min\{\rho(d_i)\:\; i=1,\ldots,r\}.
\]
Then
\[
\sdepth S/I= \max\{\rho(\mathcal{P})\:\; \mathcal{P}\; \text{is a partition of $P^g_{S/I}$}\}.
\]
We use this characterization of the Stanley depth and the results of the previous section to prove

\begin{Proposition}
\label{main}
For all $0\leq j\leq d= \dim S/I$ we have
\[
\sdepth S/I\geq \sdepth S/I_j.
\]
\end{Proposition}

\begin{proof}
Observe that $P^g_{S/I_j}=\{a\in P^g_{S/I}\:\; \rho(a)\leq j\}$. Let $t$ be the Stanley depth of $S/I_j$.
Then there exists a partition $\mathcal{P}\:\; P^g_{S/I_j}=\Union_{i=1}^r[c_i,d_i]$ with $\rho(\mathcal{P})=t$.
We complete the partition of $\mathcal{P}$ to a partition $\mathcal{P}'$ of $S/I$ by adding the intervals
$[a,a]$ with $a\in P^g_{S/I}\setminus P^g_{S/I_j}$. Since $\rho(a)>j$ for all $a\in P^g_{S/I}\setminus P^g_{S/I_j}$
it follows that $\rho(\mathcal{P}')=t$. Hence $\sdepth S/I\geq t$, as desired.
\end{proof}

We call an algebra of the form $S/I$ a monomial factor algebra if $I\subset S$ is a monomial ideal. Stanley's conjecture for a monomial factor algebra $S/I$ says that $\depth S/I\leq \sdepth S/I$.

\begin{Corollary}
\label{reduction}
Suppose Stanley's conjecture holds  for all Cohen--Macaulay monomial factor algebras  of dimension $t$. Then the conjecture holds for all monomial factor algebras of depth $t$. In particular, Stanley's conjecture holds for all monomial factor algebras  if and only
it holds for all Cohen--Macaulay monomial factor algebras.
\end{Corollary}

\begin{proof}
Let $S/I$ be a monomial factor algebra with $t= \depth S/I$. Then $S/I_t$ is Cohen--Macaulay of dimension $t$, see Corollary~\ref{hibi}.
Our assumption implies that $\sdepth S/I_t=t$. Thus the assertion follows from Proposition~\ref{main}.
\end{proof}

As a concrete application we have

\begin{Corollary}
\label{concrete}
Let $S/I$ be a monomial factor algebra with  $\depth S/I \leq 1$. Then $S/I$ satisfies Stanley's conjecture.
\end{Corollary}

\begin{proof}
According to Corollary~\ref{reduction} it suffices to show that
any Cohen-Macaulay $K$-algebra $S/I$ of dimension $\leq 1$
satisfies Stanley's conjecture. This is trivially the case if
$\dim S/I=0$, and has been shown if $\dim S/I=1$ in
\cite[Corollary 3]{A}.
\end{proof}

We now prove a statement which in a certain sense is dual to that of Corollary~\ref{reduction}.
Let $I\subset J$ be monomial ideals and $\mathcal D=\Dirsum_{i=1}^r x^{c_i}K[Z_i]$ a Stanley
decomposition of $J/I$. The number $\max\{|c_i|\:\; i=1,\ldots,r\}$  is called the
$h$-regularity of $\mathcal D$,  denoted by $\hreg(\mathcal D)$.

We
set
\[
\hreg(J/I)=\min\{\hreg( \mathcal{D})\:\;{\mathcal D}\; \text{is a
Stanley decomposition of $J/I$}\}
\]
and call this number the $h$-regularity of $J/I$. In \cite{So},
the second author conjectured that $\hreg(J/I)\leq\reg(J/I)$.

Let $g\in \NN^n$ with $g\geq a$ for all minimal monomial
generators $x^a$ of $I$ and $J$. For a partition $\mathcal{P}\:\;
P^g_{J/I}=\Union_{j=1}^r[c_i,d_i]$,  we set
\[
\sigma(\mathcal{P})=\max\{\sigma_i(\mathcal{P})\:\; i=1,\ldots,r\},
\]
where
\[
\sigma_i(\mathcal{P})=\max\{|c|\:\; c\in [c_i,d_i] \text{ and $c(j)=c_i(j)$ for all $j$ with $x_j\in Z_{d_i}$}\}.
\]

\begin{Proposition}
\label{hreg}
$\hreg(J/I)=\min\{\sigma(\mathcal{P})\:\; {\mathcal P}\; \text{is a
partition  of $P^g_{J/I}$}\}.$
\end{Proposition}

\begin{proof}
To each partition $\mathcal{P}\:\; P^g_{J/I}=\Union_{j=1}^r[c_i,d_i]$ belongs a Stanley decomposition $\mathcal{D(P)}$, as described in  \cite[Theorem 2.1(a)]{HVZ}. The assignment is such that $\hreg(\mathcal{D(P)})= \sigma(\mathcal{P})$. This shows that $\hreg(J/I)\leq \min\{\sigma(\mathcal{P})\:\; {\mathcal P}\; \text{is a
partition  of $P_{J/I}^g$}\}.$

In order to prove that equality holds, we need to find a partition $\mathcal{P}$ with $\hreg(J/I)=\sigma({\mathcal{P}})$. Let $\mathcal{D}\:\; J/I=\Dirsum_{i=1}^rx^{c_i}K[Z_i]$ be a Stanley decomposition of $J/I$ with $\hreg(\mathcal{D})=\hreg(J/I)$.   In \cite[Theorem 2.1(b)]{HVZ} it is shown  $\mathcal{P}\:\; P_{J/I}^g=\Union_{i, c_i\leq g}[c_i,d_i]$ is a partition of $P_{J/I}^g$, where  $d_i(j)=c_i(j)$ if $x_j\not\in Z_i$, and $d_i(j)=g(j)$ otherwise. Thus we see that $\hreg(J/I)=\max\{|c_i|\:\; i=1,\ldots,r\}= \sigma(\mathcal{P})$.
\end{proof}

Observe that the preceding proposition implies in particular that $\hreg(J/I)$ can be computed in a finite number of steps.

\medskip
For a graded ideal $I$ we denote by $I_{\geq j}$ the $j$th truncation of $I$, that is, the ideal generated by all homogeneous elements $f\in I$ with $\deg f \geq j$.

\begin{Proposition}
\label{ali}
For all $j\geq 0$ we have $\hreg(I)\leq \hreg(I_{\geq j})$.
\end{Proposition}

\begin{proof}
We choose $g\in \NN^n$ such that $g\geq a$ for all generators $x^a$ of $I$ and $I_{\geq j}$. Let $\mathcal{P}$ be a partition of $P^g_{I_{\geq j}}$ with $\sigma(\mathcal{P})=\hreg(I_{\geq  j})$.  We complete the partition $\mathcal{P}$ to a partition $\mathcal{P}'$ of $P^g_{I}$ by adding the intervals $[a,a]$ with $a\in P^g_{I}\setminus P^g_{I_{\geq j}}$. Then Proposition~\ref{hreg} implies that $\hreg(I)\leq \sigma(\mathcal{P'})=\sigma(\mathcal{P})=\hreg(I_{\geq j})$.
\end{proof}

\begin{Corollary}
\label{new}
Suppose $\hreg(I)\leq\reg(I)$ for all ideals with linear resolution. Then this inequality is valid for all graded ideals.
\end{Corollary}

\begin{proof}
By a result of Eisenbud and Goto \cite{EG} (see also \cite[Theorem 4.3.1]{BH}) one has $$\reg(I)=\min\{j\:\; \text{$I_{\geq j}$ has a linear resolution}\}.$$
 We choose a  $j$ such that $I_{\geq j}$ has a linear resolution. Then  our assumption and Proposition~\ref{ali} imply that
$\hreg(I)\leq \hreg(I_{\geq j})\leq \reg(I_{\geq j})=\reg(I)$.
\end{proof}

For a monomial ideal $I$ with linear resolution, minimally generated  by $x^{a_1},\ldots, x^{a_r}$, Soleyman-Jahan's conjecture reads as follows: there exist $b_1,\ldots,b_r\in P^g_{I}$ such that $P^g_{I}=\Union_{i=1}^r[a_i,b_i]$. In other words, $P^g_{I}$ can be partitioned by intervals whose lower ends correspond to the generators of $I$.

\end{document}